\documentclass[11pt]{amsart}
\usepackage{xcolor, hyperref}
\usepackage{esint} 
\usepackage{amsmath}
\usepackage{verbatim} 

\newtheorem{theorem}{Theorem}[section]
\newtheorem{lemma}[theorem]{Lemma}
\newtheorem{proposition}[theorem]{Proposition}

\DeclareMathOperator{\Ric}{Ric}

\DeclareMathOperator{\tr}{tr}
\DeclareMathOperator{\Real}{Re}
\DeclareMathOperator{\divergente}{div}

\theoremstyle{remark} 
\newtheorem{remark}{Remark}

\numberwithin{equation}{section}

\title[Area-charge inequalities for free boundary MOTS ]{Area-charge inequalities and local rigidity of free boundary MOTS in charged initial data sets}

\author[Ferreira]{Edilson Ferreira}
\author[Nunes]{Ivaldo Nunes}

\address{
Departamento de Matemática, Universidade Federal do Maranhão,
 São Luís - Brazil}
\email{ivaldo.nunes@ufma.br}
\address{Faculdade de Matemática, Universidade Federal do Pará - UFPA, Belém - Brazil}
\email{edilsonfilho@ufpa.br}

\date{\today}

\usepackage{marginnote}

\begin{document}

\begin{abstract}
In this work, we prove area-charge inequalities for free boundary MOTS in initial data sets for the Einstein-Maxwell equations with vanishing magnetic fields.  In addition, we prove a local rigidity result under the assumption that equality holds. 
\end{abstract}

\maketitle
\section{Introduction}

In the late 1970s, Schoen and Yau began investigating Riemannian manifolds with positive scalar curvature by using minimal surfaces techniques (see, for example, \cite{SY2,SY3,SY1}). A key observation made by Schoen and Yau relating these two concepts was the role played by scalar curvature in the stability inequality for minimal hypersurfaces after a suitable rearrangement using the Gauss equation. In particular, as a consequence of this idea, we have that any orientable compact stable minimal surface $\Sigma$ in an orientable Riemannian three-manifold $(M^3,g)$ must satisfy the following inequality
\begin{equation}\label{ineqsy}
\frac{1}{2}\int_\Sigma\left(R_M+|A|^2\right)\,da \leq \int_\Sigma K_\Sigma\,da=2\pi \chi(\Sigma),
\end{equation}
where $R_M$, $A$ and $K_\Sigma$ denote the scalar curvature of $(M^3,g)$, the second fundamental form and the Gauss curvature of $\Sigma$, respectively. Inequality \eqref{ineqsy} imediately implies the following:

\begin{proposition}[Schoen and Yau, \cite{SY3}]\label{propsy1}
If $(M^3,g)$ is an orientable Riemannian three-manifold with positive scalar curvature $R_M>0$, then any orientable compact stable minimal surface $\Sigma\subset M$ is homeomorphic to a two-sphere $\mathbb{S}^2$.
\end{proposition}

In combination with an existence result for area-minimizing surfaces,  Schoen and Yau were able to prove the interesting fact that the three-torus $\mathbb{T}^3$ does not admit a Riemannian metric of positive scalar curvature.  
In \cite{SY2}, they extended this for any $n$-torus $\mathbb{T}^n$ with dimension $3\leq n\leq 7$. The general case was settled by Gromov and Lawson \cite{GL1,GL2} using spin techniques. Another example of an important result that Schoen and Yau proved by using this relationship between scalar curvature and minimal surfaces is the positive mass theorem (see \cite{SY1}). It is worth mentioning that the positive mass theorem was proved by Witten \cite{Witt} in the spin case without restriction on the dimension. For recent developments in any dimension  using minimal hypersurface techniques, we refer the reader to \cite{SY4}.

More generally, as a refinement of Proposition \ref{propsy1}, it was proved by Fischer-Colbrie and Schoen \cite{FCS} that if $R_M\geq 0$, then inequality \eqref{ineqsy} implies that $\Sigma$ has to be topologically equivalent to a two-sphere or a two-torus, with this last case occurring only if $\Sigma$ is a totally geodesic flat two-torus and $R_M=0$ on $\Sigma$. This infinitesimal rigidity obtained by them naturally inspired the following conjecture (see, for example, \cite{FCS}, Remark 4, and \cite{CG1}): \textit{Let $(M^3,g)$ be a Riemannian manifold with $R_M\geq 0$ and let $\Sigma\subset M$ be a two-sided two-torus.   If $\Sigma$ is locally area-minimizing, then $M$ is flat in a neighborhood of $\Sigma$.}

The above conjecture was solved by Cai and Galloway in \cite{CG1}. In particular, as a consequence, we have that $(M,g)$ locally splits as a standard product $((-\varepsilon,\varepsilon)\times \Sigma,dt^2+g_\Sigma)$ in a neighborhood of $\Sigma$, where $g_\Sigma$ is the induced metric on $\Sigma$, which is flat. An analogous result in higher dimensions was proved by Cai \cite{Cai1}. For a simplified proof, see \cite{Gal1}.

Inspired by the results above, Bray, Brendle and Neves \cite{BBN1} proved that if $(M^3,g)$ has scalar curvature bounded from below by a positive constant $2\Lambda>0$ and $\Sigma\subset M$ is a stable minimal two-sphere, then its area satisfies $|\Sigma|\leq 4\pi/\Lambda$. Furthermore, if $\Sigma$ is locally area-minimizing and $|\Sigma|=4\pi/\Lambda$, then $M$ splits isometrically as a product $((-\varepsilon,\varepsilon)\times \Sigma, dt^2+g_\Sigma)$ in a neighborhood of $\Sigma$, where $g_\Sigma$ denotes the induced metric on $\Sigma$, which in this case has constant Gaussian curvature equal to $\Lambda$. The analogous result in the case where the scalar curvature of $(M^3,g)$ is bounded from below by a negative constant $2\Lambda<0$ and $\Sigma$ is an orientable compact locally area-minimizing surface of genus $g(\Sigma)\geq 2$ was proved by the second author \cite{Nun1}. We collect the main results proved by Cai and Galloway \cite{CG1}, Bray, Brendle and Neves \cite{BBN1} and the second author \cite{Nun1} in the following:

\begin{theorem}\label{splits1}
Let $(M^3,g)$ be an orientable Riemannian three-manifold with scalar curvature bounded satisfying $R_M\geq 2\Lambda$, where $\Lambda\in\mathbb{R}$. Let $\Sigma\subset M$ be an orientable compact minimal surface. Suppose that $\Sigma$ is locally area-minimizing. Then:
\begin{itemize}
\item[(1)] If $\Lambda=0$ and $\Sigma$ is a two-torus, then $\Sigma$ is flat with the induced metric $g_\Sigma$ and $(M^3,g)$ splits isometrically as a standard product $((-\varepsilon,\varepsilon)\times \Sigma, dt^2+g_\Sigma)$ in a neighborhood of $\Sigma$ (see \cite{CG1}).
\item[(2)] If $\Lambda>0$, then $\Sigma$ is a two-sphere and the area of $\Sigma$ satisfies $|\Sigma|\leq 4\pi/\Lambda$. Moreover, if equality holds then $\Sigma$ has constant Gaussian curvature equal to $\Lambda$ with the induced metric $g_\Sigma$ and  $(M^3,g)$ splits isometrically as a standard product $((-\varepsilon,\varepsilon)\times \Sigma, dt^2+g_\Sigma)$ in a neighborhood of $\Sigma$ (see \cite{BBN1}).
\item [(3)] If $\Lambda<0$ and $\Sigma$ has genus greater than $1$, then the area of $\Sigma$ satisfies $|\Sigma|\geq 2\pi\chi(\Sigma)/\Lambda$. Moreover, if equality holds then $\Sigma$ has constant Gaussian curvature equal to $\Lambda$ with the induced metric $g_\Sigma$ and  $(M^3,g)$ splits isometrically as a standard product $((-\varepsilon,\varepsilon)\times \Sigma, dt^2+g_\Sigma)$ in a neighborhood of $\Sigma$ (see \cite{Nun1}).
\end{itemize}
\end{theorem}

It is worth noting that Micallef and Moraru gave a unified proof of the local splitting results obtained in \cite{CG1}, \cite{BBN1} and \cite{Nun1}, and that Moraru \cite{Mor1} was able to generalize the result in \cite{Nun1} to higher dimensions. Many splitting results, in the same spirit as those previously cited, have been proved in recent years. We refer the reader, for example, to \cite{MaxNun}, \cite{CasRos}, \cite{Ambrozio}, \cite{BCBS1}, \cite{BaBaCruz}, \cite{MazRos}, \cite{EspRos1}, \cite{ChEMo}, \cite{Mendes2019}, \cite{VanLima1}, \cite{BarCruz}, \cite{BarCon1}, \cite{BaRoLi}, \cite{LeeParkPyo}, \cite{PeVeVi} and the references therein.

An important feature of minimal surfaces, among others, is that they arise naturally in the context of general relativity. In fact, in time-symmetric (totally geodesic) initial data sets, the concept of a minimal surface coincides with that of a marginally outer trapped surface (MOTS, for short), which plays a crucial role in the study of black holes in general relativity. Moreover, in this time-symmetric setting, the condition $R_M\geq 2\Lambda$ corresponds precisely to the so-called dominant energy condition, where $\Lambda\in\mathbb{R}$ represents the cosmological constant. Therefore, the results stated in Theorem \ref{splits1} may be viewed as local rigidity results for initial data sets in the time-symmetric case. Motivated by this analogy, many similar results in the non time-symmetric setting have been established in recent years. For example, analogues of items (1), (2) and (3) of Theorem \ref{splits1} were obtained by Galloway \cite{Gal1}, Galloway and Mendes \cite{GalMen1}, and Mendes \cite{mendes2019rigidity}, respectively. We note that Mendes \cite{mendes2019rigidity} also proved a non time-symmetric version of the rigidity theorem proved by Moraru in \cite{Mor1}. 

In this work, we are interested in the case where the initial data sets are endowed with an electric vector field. More precisely, we focus on initial data sets for the Einstein-Maxwell equations with vanishing magnetic field. Recently, Galloway and Mendes \cite{GalMen2} and Lima, Sousa, and Batista \cite{patoMOTS} obtained sharp area bounds for MOTS in this setting. Also, Mendes \cite{Mendescarga} proved that some area-charge inequalities obtained previously by Gibbons \cite{Gibb} and Dain, Jaramillo, and Reiris \cite{DainJaramilloReiris} in the time-symmetric case are sharp. He also generalized this to the non time-symmetric case.  We also note that Cruz and Mendes \cite{cruzmendescarga} obtained additional area-charge inequalities and rigidity results for time-symmetric initial data sets for the Einstein-Maxwell equations with vanishing magnetic field. 

Our main goal here is to generalize the area-charge inequalities and certain rigidity results obtained in \cite{GalMen2}, \cite{patoMOTS} and \cite{cruzmendescarga} to the context of free boundary MOTS  in initial data sets with boundary satisfying both the interior and boundary dominant energy conditions. For precise definitions, we refer the reader to Section \eqref{Prelsec}. We note that free boundary MOTS (more generally, capillary MOTS) were introduced by Alaee, Lesourd and Yau in \cite{alaee2021stable}, where they defined the notion of stability and proved some local rigidity results in this setting. Recently, further local rigidity theorems for free boundary MOTS were obtained by Mendes \cite{mendes2022rigidity}  and Almeida and Mendes \cite{AlmeidaMendes}. Our first result is the following:

\begin{theorem}[Area-charge inequalities]\label{areacharge} Let $(M^3,g,K,E)$ be  a 3-di\-men\-sion\-al initial data set for the Einstein-Maxwell equations with absence of magnetic field  with boundary $\partial 
M\neq \emptyset$. Suppose that $E(p)\in T_p(\partial M)$ for all $p\in\partial M$ and that $(M^3,g,K,E)$ satisfies the dominant energy conditions (DEC) (see \eqref{dec}) and (BDEC) (see \eqref{bdec}). Let $\Sigma^2 \subset M^3$ be a two-sided compact free boundary MOTS. If $\Sigma$ is stable, then  the charge and area of $\Sigma$ satisfy:
\begin{itemize}
\item[(i)] If $\Lambda>0$, then $\chi(\Sigma)=1$, $4\Lambda \mathfrak{q}(\Sigma)^2\leq 1$ and 
\begin{equation}\label{ineqlambdapos}
\dfrac{\pi}{\Lambda}\left(1-\sqrt{1-4\Lambda\mathfrak{q}(\Sigma)^2}\right)\leq |\Sigma| \leq \dfrac{\pi}{\Lambda}\left(1+\sqrt{1-4\Lambda\mathfrak{q}(\Sigma)^2}\right).
\end{equation}
\item[(ii)] If $\Lambda=0$, then $\chi(\Sigma)\geq 0$ and
\begin{equation}\label{ineqlambdazero}
4\pi^2\mathfrak{q}(\Sigma)^2\leq 2\pi \chi(\Sigma)|\Sigma|.
\end{equation}
In particular, if $\chi(\Sigma)=0$, then the charge of $\Sigma$ is zero, that is, $\mathfrak{q}(\Sigma)=0$, and if $\chi(\Sigma)=1$, then
\begin{equation}
|\Sigma|\geq 2\pi \mathfrak{q}(\Sigma)^2.
\end{equation}
\item[(iii)] If $\Lambda<0$, then
\begin{equation}\label{ineqlambdaneg}
|\Sigma|\geq \dfrac{\pi}{|\Lambda|}\left(-\chi(\Sigma)+\sqrt{\chi(\Sigma)^2+4|\Lambda|\mathfrak{q}(\Sigma)^2}\right).
\end{equation}
\end{itemize}
Here, $\mathfrak{q}(\Sigma)$ and $\chi(\Sigma)$ denote the electric charge and the Euler characteristic of $\Sigma$, respectively.
\end{theorem}

Next, by assuming some additional condition on $\Sigma$, $E$ and $K$, and that equality holds in some of the inequalities \eqref{ineqlambdapos}, \eqref{ineqlambdazero} and \eqref{ineqlambdaneg}, we prove the following local rigidity result:

\begin{theorem}[Local rigidity]\label{main}
Let $(M^3,g,K,E)$ be an initial data set as in Proposition \ref{areacharge}. Suppose, in addition, that the electric vector field $E$ is diverge-free, that is, $\divergente E=0$ and that $K$ is $2$-convex. Suppose that $\Sigma^2\subset M$ is a weakly outermost properly embedded compact free boundary MOTS. In addition, in the case $\Lambda> 0$, suppose that $\Sigma$ is outer area minimizing. Then all the area-charge inequalities \eqref{ineqlambdapos}, \eqref{ineqlambdazero} and \eqref{ineqlambdaneg} hold, according to the sign of $\Lambda$. Moreover,  if equality is attained, then there is a  neighborhood $V\cong [0,\varepsilon)\times \Sigma$ of $\Sigma$ in $M$ such that on $V$ we have:
\begin{itemize}
\item[(i)] $E=a N_t$ for some $a\in\mathbb{R}$, where $N_t$ is the unit normal vector field to $\Sigma_t\cong\{t\}\times \Sigma$ pointing in the direction of the foliation;
\item[(ii)] $(\Sigma,\gamma)$, where $\gamma$ is the metric on $\Sigma$ induced by $g$, has constant Gauss curvature $K_\Sigma=\Lambda+a^2$ and geodesic boundary $\partial\Sigma$, that is, $k_g=0$. In particular, if $\Lambda=0$ and $\Sigma$ is an annulus (i.e., $\chi(\Sigma)=0$), then $E=0$ on $V$. Moreover, the metric $g$ can be written on $V$ as $g=dt^2+\gamma$.
\item[(iii)] $K=f(t)dt^2$ on $V$ for some smooth function $f$ depending only on $t$;
\item[(iv)] $J=0$ on $V$;
\item[(v)] The dominant energy condition (DEC) saturates on $V$, that is, $\mu=\Lambda+a^2$ on $V$, and the boundary energy condition (BDEC) saturates on $\partial M\cap V$, that is, $H^{\partial M} = |(i_{\bar{N}}\pi)^T|=0$ on $\partial M\cap V$.
\end{itemize}
\end{theorem}

\begin{remark}
We note that, in the case $\Lambda=0$ and $\Sigma$ is an annulus, the local rigidity in Theorem \ref{main} can be obtained as a consequence of Theorem 1.4 of Mendes \cite{mendes2022rigidity}.
\end{remark}

This paper is organized as follows. In Section \ref{Prelsec} we present preliminary definitions and results  concerning initial data sets, MOTS and stability that will be used throughout the paper. In Section \ref{areachargesec}, we prove the area-charge estimates for stable MOTS presented in Theorem \ref{areacharge}. Finally, in Section \ref{mainproofsec} we give the proof of the local rigidity theorem \ref{main}.

\section{Preliminaries}\label{Prelsec}

\subsection{Basic definitions}
An \textit{initial data set for the Einstein-Maxwell equations with absence of magnetic field} is a tuple $(M^{n+1},g,K,E)$, where $(M,g)$ is a smooth Riemannian manifold (possibly with nonempty boundary), $K$ is a smooth symmetric $(0,2)$-tensor and $E\in \mathfrak{X}(M)$ is a smooth vector field representing the electric field. 

The following function and 1-form 
\begin{equation}\label{muJ}
\mu=\dfrac{1}{2}(R_M-|K|^2+(\tr K)^2) \ \ \mbox{and} \ \ J= \divergente(K-(\tr K)g),
\end{equation}
where $R_M$ stands for the scalar curvature of $(M,g)$, are the \textit{local energy density}  and \textit{local current density} of the initial data set, respectively.

It is said that $(M^{n+1},g,K,E)$ satisfies the \textit{dominant energy condition (DEC)} if 
\begin{equation}\label{dec}
\mu-|J|\geq \Lambda + |E|^2,    
\end{equation}
where $\Lambda\in\mathbb{R}$ is a constant representing the cosmological constant.

Let $\Sigma^n\subset M^{n+1}$ be a smooth connected two-sided hypersurface. Denote by $N$ the unit normal vector field to $\Sigma$, which is unique up to a sign. We will denote the second fundamental form and the mean curvature of $\Sigma$ with respect to $N$ by $A$ and $H$, respectively. Here we are adopting the following convention: $A(X,Y)=g(\overline{\nabla}_XN,Y)$ for all $X,Y\in T\Sigma$, where $\overline{\nabla}$ denotes the Riemannian connection of $(M,g)$.

The \textit{null second fundamental forms} of $\Sigma$ in $(M,g,K,E)$ are given by
\begin{equation}\label{null2ndff}
\chi^\pm=K\pm A. 
\end{equation}
Its traces  are called the \textit{null mean curvatures} of $\Sigma$. More precisely, we have 
\begin{equation}\label{nullmeancurvature}
\theta^\pm=\tr_\Sigma K\pm H.
\end{equation}

As usual in the literature, we will use the following terminology:
\begin{itemize}
\item $\Sigma$ is an \textit{outer trapped hypersurface} if $\theta^+< 0$;
\item $\Sigma$ is a \textit{weakly outer trapped hypersurface} if $\theta^+\leq0$;
\item $\Sigma$ is a \textit{marginally outer trapped hypersurface} (\textit{MOTS}, for short) if $\theta^+=0$. 
\end{itemize}

Now consider a smooth variation $t\in (-\varepsilon,\varepsilon)\mapsto \Sigma_t\subset M$, $\Sigma_0=\Sigma$, of $\Sigma$ in $M$. Let $V=\frac{\partial }{\partial t}|_{t=0}$ be the variational vector field of the variation along $\Sigma$ and suppose that $V=\varphi N$, for some $\varphi\in C^\infty(\Sigma)$. For each $t\in (-\varepsilon,\varepsilon)$, let $N_t$ be the unit normal vector along $\Sigma_t$, depending smoothly on $(t,x)$ and such that $N_0=N$, and let $\theta^+(t)$ denote the null mean curvature of $\Sigma_t$ with respect to $N_t$. In this case, it is well-known that
\begin{equation}\label{variationoftheta}
(\theta^+)^\prime(0)=\dfrac{\partial}{\partial t}\theta^+(t,\cdot)|_{t=0}= L\varphi + \left(\theta^+\tau -\dfrac{1}{2}(\theta^+)^2\right)\varphi,
\end{equation} 
where  
\begin{equation}\label{stabilityoperator}
L\varphi=-\Delta\varphi + 2\langle X,\nabla \varphi\rangle + \left(Q-|X|^2+\divergente  X\right)\varphi, 
\end{equation}
$\tau = \tr K$ and 
\begin{equation}\label{defofQ}
Q=\dfrac{1}{2}R^\Sigma - (\mu+J(N))-\dfrac{1}{2}|\chi^+|^2.
\end{equation}
In the above, $\Delta$, $\nabla$ and  $R_\Sigma$ denotes the Laplacian operator, the gradient and  the scalar curvature of $\Sigma$ with respect to the induced metric, respectively, and $X$ is the vector field tangent to $\Sigma$ dual to the 1-form $K(N,\cdot)$.

When $\Sigma$ is a MOTS, the operator $L$ given by \eqref{stabilityoperator} is usually called the \textit{stability operator of $\Sigma$}. This terminology is motivated by the time-symmetric case (i.e., $K=0$). In this situation, a MOTS is just a minimal hypersurface, and $L$ coincides with 
$$
-\Delta-(\Ric(N,N)+|A|^2),
$$ 
which is the classical stability operator associated with minimal hypersurfaces.  The notion of stability for MOTS will be discussed in detail in Subsection \ref{stabilityofmots}.

\subsection{Further definitions in case where $\partial M\neq\emptyset$ and $\partial\Sigma \neq \emptyset$}

In this work, we are interested in compact MOTS with free boundary in initial data sets $(M,g,K,E)$ with $\partial M\neq \emptyset$. In this subsection, we present the additional definitions and dominant energy condition in this setting. 

Suppose from now on that $\partial M\neq \emptyset$. Let $\bar{N}$ be the outward unit vector field of $\partial M$ in $M$ ad let $\Pi^{\partial M}(X,Y)=\langle \overline{\nabla}_X\bar{N},Y\rangle$, $X,Y\in T(\partial M)$, be the second fundamental form of $\partial M$ in $M$ with respect to $\bar{N}$. The mean curvature of $\partial M$ with respect to $\bar{N}$, that is, the trace of $\Pi^{\partial M}$, will be denoted by $H^{\partial M}$.

We say that $(M^{n+1},g,K,E)$ satisfies the \textit{boundary dominant energy condition (BDEC)} if
\begin{equation}\label{bdec}
H^{\partial M}\geq |(i_{\bar{N}}\pi)^\top|,
\end{equation}
where $\pi=K-(\tr K)g$ is the \textit{momentum tensor} of $(M,g,K,E)$ and $(i_{\bar{N}} \pi)^\top$ is the 1-form on $\partial M$ given by $(i_{\bar{N}}\pi)^\top(X)=\pi(\bar{N},X)$ for all $X\in T(\partial M)$. Regarding the motivation for the dominant energy condition (BDEC) defined by \eqref{bdec} we refer the reader to \cite{almaraz} (see Remark 2.7), where it was first introduced by Almaraz, de Lima and Mari.

When $\partial M \neq \emptyset$, we will always assume that the electric field $E$ is tangent to $\partial M$, that is,
$E(p) \in T_p(\partial M)$ for all $p \in \partial M$.

Let $\Sigma^n \subset M^{n+1}$ be a compact two-sided hypersurface with $\partial\Sigma\neq \emptyset$ such that
$\partial \Sigma \subset \partial M$.
Let $N$ be a unit normal vector field along $\Sigma$.
Since there is no contribution to the electric charge arising from $\partial M$, we define the
\textit{charge of $\Sigma$} as
\begin{equation}\label{chargedef}
\mathfrak{q}(\Sigma)=\frac{2}{\omega_n}\int_\Sigma\langle E,N\rangle\,dv.
\end{equation}
Note that, in contrast to the case where $\Sigma$ has no boundary, here we normalize the above integral by half the volume $\omega_n$ of the unit sphere $\mathbb{S}^n \subset \mathbb{R}^{n+1}$.

\subsection{Stability of MOTS}\label{stabilityofmots}
Suppose now that $\Sigma^n\subset (M^{n+1},g,K,E)$ is  a  MOTS. 

Assume first that $\Sigma$ is closed, that is, compact without boundary. In the time-symmetric case (i.e., $K=0$), where a MOTS reduces to a minimal hypersurface, the notion of stability means, as usual, that the second variation of the volume functional is nonnegative for all normal variations of $\Sigma$ in $M$. From an analytical point of view, this is equivalent to say that the first eigenvalue $\lambda_1(L)$ of the stability operator $L=-\Delta -(\Ric(N,N)+|A_\Sigma|^2)$ is nonnegative. Moreover, it is well-known that $\lambda_1(L)\geq 0$ if and only if there is a \textit{positive} smooth function $\phi\in C^\infty(\Sigma)$ such that $L\phi \geq 0$. 

Inspired by the above, Andersson, Mars and Simon \cite{ALS} introduced a notion of stability for MOTS in the non-time-symmetric case (i.e., $K \not\equiv 0$). In this context, the stability operator $L$ given by \eqref{stabilityoperator} is not self-adjoint in general, but  it has a real eigenvalue, called the \textit{principal eigenvalue}, such that for any other eigenvalue $\lambda$ of $L$ the inequality $\lambda_1(L)\leq \Real(\lambda)$ holds. Moreover, the corresponding eigenfunction $\phi$ , $L\phi=\lambda_1(L)\phi$, is unique up to a multiplicative constant and can be chosen to be real and everywhere positive. Then, as defined in \cite{ALS}, we say that $\Sigma$ is a \textit{stable} MOTS if $\lambda_1(L)\geq 0$, in analogy with the time-symmetric case. In \cite{ALS},  it is proved the condition $\lambda_1(L)\geq 0$ is equivalent to the existence of a \textit{positive} function $\phi\in C^\infty(\Sigma)$ such that $L\phi\geq 0$.

Recently, Alaee, Lesourd and Yau \cite{alaee2021stable}, defined a notion of stability for capillary MOTS in initial data sets with nonempty boundary. 

Let $\Sigma^n\subset (M^{n+1},g,K,E)$ be a MOTS and suppose that $\partial M\neq\emptyset$ and $\partial\Sigma\neq\emptyset$. We say that $\Sigma$ is \textit{capillary} if $\Sigma$ meets $\partial M$ at a constant angle $\theta\in (0,\pi)$ along $\partial\Sigma$. If $\nu$ denotes the unit conormal of $\partial\Sigma$ in $\Sigma$, then this condition is equivalent to say that $\partial\Sigma\subset \partial M$ and $\langle \bar{N},\nu\rangle=\sin \theta$ on $\partial\Sigma$. In the case $\theta=\pi/2$, we call $\Sigma$ a \textit{free boundary MOTS} and this means that $\nu=\bar{N}$ and $N\in T(\partial M)$ on $\partial\Sigma$.  We will focus here in the free boundary case. Following \cite{alaee2021stable}, we say that $\Sigma$ is a \textit{stable free boundary MOTS} if there is a smooth function $\phi\geq 0$, $\phi\not\equiv 0$, defined on $\Sigma$ such that
\begin{equation}
\left\{
\begin{array}{rc}
L\phi  \geq  0 &  \mbox{in $\Sigma$}\\
B\phi=\dfrac{\partial \phi}{\partial\nu}-\Pi^{\partial M}(N,N)\phi =  0 & \mbox{on $\partial \Sigma$}.
\end{array}
\right.
\end{equation}
For details and motivation about the above definition we refer the reader to \cite{alaee2021stable}.

Finally, suppose that $\Sigma \subset (M,g,K,E)$ is a free boundary MOTS that separates $M$ into two regions $M^-$ and $M^+$. The notation is chosen such that the unit normal vector $N$ points toward $M^+$. We call $M^+$ the exterior of $\Sigma$. We say that $\Sigma$ is \textit{weakly outermost} if there is no outer trapped $(\theta^+<0)$ free boundary hypersurface $\widehat{\Sigma} \subset M^+$ homologous to $\Sigma$. 

\begin{remark}\label{woimpliesstable}
It is well known that a closed weakly outermost MOTS $\Sigma$ in an initial data set $(M^3,g,K)$ is stable. In fact, if $\Sigma$ is not stable, then $\lambda_1(L)<0$, where $\lambda_1(L)$ is the principal eigenvalue of the stability operator $L$ of $\Sigma$. Thus, if $\Sigma_t$, $t\in[0,\varepsilon)$,  is a variation of $\Sigma=\Sigma_0$ in $M$ with the variational vector field equal to $V=\phi N$, where $\phi>0$ is an eigenfunction of $\phi$ associated with $\lambda_1(L)$, then $(\theta^+)^\prime(0)=L\phi=\lambda_1(L)\phi<0$. Therefore, $\theta^+(t)<0$ for all $t>0$ sufficiently small. This would imply that $\Sigma$ is not weakly outermost. 

In the free boundary case, there is no general existence result for the principal eigenvalue $\lambda_1(L)$ of the stability operator with Robin boundary condition $B\phi=0$. Therefore, we cannot use the above argument in general in order to prove that a weakly outermost \textit{free boundary} MOTS is stable. It would be interesting to prove this is true in full generality. In \cite{alaee2021stable}, they were able to prove this with the additional assumption that $\sup_{\partial\Sigma} \Pi^{\partial M}(N,N)\leq 0$. In this case, they proved the existence of the principal eigenvalue $\lambda_1(L)$ of $L$ with Robin condition $B\phi=0$ (see Theorem 6.8 of \cite{alaee2021stable}). In particular, under this additional assumption, we can remove the stability hypothesis on $\Sigma$ from the statement of Theorem \ref{main}.
\end{remark}

\subsection{Useful lemmas}
In this subsection we list some lemmas that will be useful in the proofs of our results. 

Let $(M^{n+1},g,K,E)$ be an initial data set with $\partial M\neq \emptyset$ and let $\Sigma^{n}\subset M^{n+1}$ be a two-sided compact MOTS with free boundary. Since its stability operator $L$ (see \eqref{stabilityoperator}) is not self-adjoint in general, it is convenient to consider the ``symmetrized'' operator $L_0:C^\infty(\Sigma)\to C^\infty(\Sigma)$ given by
\begin{equation}\label{symoperator}
L_0\varphi = -\Delta \varphi + Q\varphi.
\end{equation}

Let $\lambda_1(L_0)$ be the first eigenvalue of $L_0$ with the Robin-type boundary condition $B_0\varphi=B\varphi+\langle X, \nu \rangle \varphi=0$ on $\partial\Sigma$. Note that $\lambda_1(L_0)$ is characterized variationally as the infimum of the following Rayleigh quotient
\begin{equation}\label{Rayleigh}
 \lambda_{1}(L_0) = \inf _{u\in C^{\infty}(\Sigma)\setminus \{0\}} \dfrac{\int_\Sigma\left(|\nabla u|^2+Qu^2\right)dv-\int_{\partial\Sigma} \left(\Pi^{\partial M}(N,N)-\langle X,\nu\rangle\right)u^2da}{\int_\Sigma u^2\,dv},
\end{equation}
where the infimum is taken over all $u\in C^\infty(\Sigma)\setminus\{0\}$.

The following lemma from  \cite{mendes2022rigidity} tells us that $\lambda_1(L_0)\geq 0$ if $\Sigma$ is stable.

\begin{lemma}\label{estimativeeigenvalue}
Let $\Sigma^n\subset (M^{n+1},g,K,E)$ be two-sided compact MOTS with free boundary. If $\Sigma$ is stable, then $\lambda_1(L_0)\geq 0$.
\end{lemma}
\begin{proof}
We refer the reader to \cite{mendes2022rigidity}.
\end{proof}

We finish this subsection with the following useful lemma from \cite{mendes2019rigidity}.

\begin{lemma}\label{Lemma f nonnegative}
    Let $f \in C^{1}[0,\varepsilon)$ and $\eta , \xi , \rho \in C^{0}([0,\varepsilon )$ be functions satisfying $\max\{f, \rho \} \geq 0, \xi\geq 0, \eta >0,$ and $f(0)=0$. If
    $$f'(t)\eta (t) - f(t)\rho (t) \leq \int_{0}^{t}f(s)\xi(s)ds, \ \ \forall\, t\in [0,\varepsilon), $$
    then $f\leq 0$. In particular, if $f$ is nonnegative, then $f\equiv 0.$
\end{lemma}

\section{Area-charge inequalities and infinitesimal rigidity}\label{areachargesec}
We begin this section with the proof of Proposition \eqref{areacharge}.
\begin{proof}[Proof of Proposition \ref{areacharge}]
Since $\Sigma $ is a stable MOTS with free boundary, it follows from Lemma \eqref{estimativeeigenvalue} that $\lambda_1(L_0)\geq 0$. Thus,  by Rayleigh's formula \eqref{Rayleigh} we  have 

\begin{equation}
     0 \leq \int_{\Sigma}(|\nabla u|^2 + Qu^2)\,da - \int_{\partial \Sigma}(\Pi^{\partial M}(N,N) - \langle X,\nu \rangle )u^2\,d\ell,
\end{equation}
for all $u\in C^\infty(\Sigma)$. 
Taking $u \equiv 1$, we obtain
\begin{align}\label{estimative area}
  0  
    \leq& \int_{\Sigma}Q\,da - \int_{\partial \Sigma}(\Pi^{\partial M}(N,N) - \langle X,\nu \rangle )\,d\ell \\
    =& \int_{\Sigma}(K_{\Sigma} - (\mu + J(N)) - \frac{1}{2}|\chi^{+}| ^2)\,da\nonumber\\
    &- \int_{\partial \Sigma}(\Pi^{\partial M}(N,N)
    - \langle X,\nu \rangle) \,d\ell\nonumber,
\end{align}
where we have used that $Q=K_\Sigma -(\mu+J(N)) -\dfrac{1}{2}|\chi^+|^2$. Here $K_\Sigma$ stands for the Gauss curvature of $\Sigma$.

Since $\Sigma$ is a free boundary surface, we have, 
\begin{align}\label{free boundary condition}
    H^{\partial M} = k_g + \Pi^{\partial M}(N,N)
\end{align}
where $k_g$ denotes the geodesic curvature of $\partial \Sigma$ in $\Sigma$ and $H^{\partial M}$ denotes the mean curvature of $\partial M$.  Furthermore, we have
\begin{equation}\label{iota subs}
    (\iota_{\bar{N}}\pi)^{\top}(N) = K(\bar{N},N) - (\tr K)\langle \bar{N},N\rangle 
    = \langle X,\nu \rangle.
\end{equation}
Substituting \eqref{free boundary condition} and \eqref{iota subs} in \eqref{estimative area} we obtain
\begin{equation}\label{auxneq2}
\begin{array}{rcl}
    0 
    &\leq & \int_{\Sigma}\left(K_\Sigma - (\mu + J(N)) - \frac{1}{2}|\chi^{+}|^2\right)\,da\vspace{0.2cm}
    \\
    &&- \int_{\partial \Sigma}\left(H^{\partial M} - k_g - (\iota _{\bar{N}}\pi)^{\top}(N) \right)\,d\ell\vspace{0.2cm}   \\ 
    &\leq& \int_{\Sigma}K_{\Sigma}\,da + \int_{\partial \Sigma}k_gds - \int_{\Sigma}\left(\mu + J(N)\right)\,da \vspace{0.2cm}
    \\
    &&- \int_{\partial \Sigma}\left(H^{\partial M} - (\iota_{\bar{N}}\pi)^{\top}(N)\right)\,d\ell .
    \end{array}
\end{equation}
Since $\mu + J(N) \geq \mu - |J| \geq  \Lambda + |E|^2$ in $\Sigma$ and $H^{\partial M} \geq |(\iota _{\bar{N}}\pi )^{\top}|\geq (\iota_{\bar{N}}\pi)^\top(N)$ on $\partial\Sigma$, and invoking the Gauss-Bonnet theorem, we get that
\begin{equation}\label{auxineq}
    0  \leq 2\pi \chi(\Sigma)  -  \Lambda|\Sigma| - \int_{\Sigma}|E|^{2}\,da .
\end{equation}
Next, by applying the Cauchy-Schwarz inequality to the definition of $\mathfrak{q}(\Sigma)$ \eqref{chargedef}, we obtain the following:
\begin{equation}\label{cauchy-schwarz}
    4\pi ^2\mathfrak{q}(\Sigma)^2 = \Big(\int_{\Sigma} \langle E,N\rangle \,da\Big)^2 
    \leq |\Sigma|\int_\Sigma\langle E,N\rangle^2\,da\leq |\Sigma|\int_{\Sigma}|E|^2\,da.
\end{equation}
Using the above inequality in \eqref{auxineq}, we have
\begin{equation}\label{ineqpolynomial2}
    0 \leq 2\pi \chi (\Sigma) - \Lambda|\Sigma|  - \dfrac{4\pi ^2\mathfrak{q}(\Sigma)^2}{|\Sigma|},
\end{equation}
which is equivalent to
\begin{equation}\label{inequalitypolynomial}
0\leq -\Lambda |\Sigma|^2+2\pi \chi(\Sigma)-4\pi^2 \mathfrak{q}(\Sigma)^2.
\end{equation}

Now, let us analyze each case depending on the sign of $\Lambda$:
\begin{itemize}
    \item If $\Lambda > 0$, we get from \eqref{inequalitypolynomial} that
\begin{equation*}
    0<\Lambda|\Sigma|\leq 2\pi \chi(\Sigma)-4\pi^2\mathfrak{q}(\Sigma)^2\leq 2\pi \chi(\Sigma).
\end{equation*}
Therefore, $\chi (\Sigma) = 1 $. Moreover, we have that $\Delta= 4\pi^2(1-4\Lambda\mathfrak{q}(\Sigma)^2)\geq 0$, that is, $4\Lambda \mathfrak{q}(\Sigma)^2\leq 1$, and 
$$
\dfrac{\pi}{\Lambda}\left(1-\sqrt{1-4\Lambda\mathfrak{q}(\Sigma)^2}\right)\leq \Sigma\leq \dfrac{\pi}{\Lambda}\left(1+\sqrt{1-4\Lambda\mathfrak{q}(\Sigma)^2}\right).
$$
\item If $\Lambda = 0$ in \eqref{inequalitypolynomial}, we get 
\begin{align*}
   4\pi ^2 \mathfrak{q}(\Sigma )^2 \leq 2\pi \chi (\Sigma) |\Sigma|. 
\end{align*}
Therefore, $\chi (\Sigma) \geq 0$, and if $\chi (\Sigma) = 0 $, then we conclude that the electric charge of $\Sigma$ is zero, i.e., $\mathfrak{q}(\Sigma) = 0$. In the case $\chi(\Sigma)=1$, we obtain that $|\Sigma|\geq 2\pi \mathfrak{q}(\Sigma)^2$.
\item If $\Lambda <0$,  we get from \eqref{inequalitypolynomial} that \begin{equation*}
    |\Sigma| \geq \dfrac{\pi}{|\Lambda|}\Big(-\chi (\Sigma) + \sqrt{\chi (\Sigma)^2 + 4|\Lambda| \mathfrak{q}(\Sigma)^2}\Big).
\end{equation*}
\end{itemize}

Therefore, we have proved inequalities \eqref{ineqlambdapos}, \eqref{ineqlambdazero} and \eqref{ineqlambdaneg}.
\end{proof}

In the next proposition, we analyze what happens when equality holds in some of inequalities \eqref{ineqlambdapos}, \eqref{ineqlambdazero} and \eqref{ineqlambdaneg}. This will be very useful in the proof of Theorem \eqref{main}.

\begin{proposition}[Infinitesimal rigidity]\label{infrigidity}
In Proposition \eqref{areacharge}, if equality in \eqref{ineqlambdapos}, \eqref{ineqlambdazero} and \eqref{ineqlambdaneg}, then 
\begin{itemize}
\item[(i)] $E=aN$, for some $a\in\mathbb{R}$;
\item[(ii)] The null second fundamental form $\chi^+$ of $\Sigma$ vanishes;
\item[(iii)] The dominant energy condition (DEC) saturates on $\Sigma$, that is, $\Lambda-|J|=\Lambda+|E|^2$ on $\Sigma$. Moreover, $|J|=-J(N)$ on $\Sigma$ and $J|_{T\Sigma}\equiv 0$;
\item[(iv)] The boundary energy condition (BDEC) saturates on $\partial\Sigma$, that is, $H^{\partial M}=|(i_{\bar{N}}\pi)^\top|$ on $\partial\Sigma$. Moreover, $|(i_{\bar{N}}\pi)^\top|=(i_{\bar{N}}\pi)^\top(N)$ and $(i_{\bar{N}}\pi)^\top|_{T(\partial\Sigma)}\equiv 0$;
\item[(v)] $\Sigma$ has constant Gauss curvature $K_\Sigma=\Lambda + |E|^2$ (which by (ii) and (iii) is equivalent to $Q=0$ on $\Sigma$) and $\partial \Sigma$ is a geodesic in $\Sigma$, that is, $k_g=0$ (which by item (iv), \eqref{free boundary condition} and \eqref{iota subs} is equivalent to $ \Pi^{\partial M}(N,N)=\langle X,\nu\rangle $ on $\partial\Sigma$).
\end{itemize}
\end{proposition}
\begin{proof}
First note that equality in \eqref{ineqlambdapos}, \eqref{ineqlambdazero} and \eqref{ineqlambdaneg} is equivalent to have equality in \eqref{inequalitypolynomial}. Therefore, all inequalities used to get \eqref{inequalitypolynomial} are in fact equalities. 

From equality in \eqref{cauchy-schwarz}, we obtain that $E=aN$ on $\Sigma$, for some constant $a\in\mathbb{R}$.

From in \eqref{auxneq2} and \eqref{auxineq} we obtain that $\chi^+=0$, $\Lambda-|J|=\Lambda+|E|^2$ and $|J|=-J(N)$ on $\Sigma$. In particular, $J|_{T\Sigma}\equiv 0$. We also obtain that $H^
{\partial M}=|(i_{\bar{N}}\pi)^\top|$ and $|(i_{\bar{N}}\pi)^\top|=(i_{\bar{N}}\pi)^\top(N)$ on $\partial\Sigma$. In particular, $(i_{\bar{N}}\pi)^\top|_{T(\partial\Sigma)}\equiv 0$. 

Next, it follows from equality in \eqref{estimative area} that the constant function $u\equiv 1$ attains the infimum in \eqref{Rayleigh}. As a consequence we have that $\lambda_1(L_0)=0$ and $Q=L_0(1)=0$ in $\Sigma$ and $-\Pi^{\partial M}(N,N)+\langle X,\nu\rangle=B_0(1)=0$ on $\partial\Sigma$. Thus, from \eqref{defofQ}, we obtain that $K_\Sigma=\Lambda+|E|^2$ on $\Sigma$. Finally, since $H^{\partial M}=(i_{\bar{N}}\pi)^\top(N)=\langle X,\nu\rangle=\Pi^{\partial M}(N,N)$ on $\partial\Sigma$, we have by \eqref{free boundary condition} that $\partial \Sigma$ is a geodesic, that is $k_g=0$.
\end{proof}

\section{Proof of Theorem \ref{main} }\label{mainproofsec}

Since $\Sigma$ is a free boundary stable MOTS, we have by Proposition \ref{areacharge} that the area of $\Sigma$ satisfies the area-charge inequalities \eqref{ineqlambdapos} if $\Lambda>0$, \eqref{ineqlambdazero} if $\Lambda=0$ and \eqref{ineqlambdaneg} if $\Lambda<0$. 

Now, suppose that one of the inequalities above is actually an equality. Then, by item (v) of Proposition \ref{infrigidity}, we have  $Q=0$ on $\Sigma$ and $\Pi^{\partial M}(N,N)=\langle X,\nu\rangle$ on $\partial\Sigma$.

Therefore, since $\Sigma$ is a free boundary stable MOTS, we have by Lemma 3.5 from \cite{mendes2019rigidity} that $\lambda=0$ is a simple eigenvalue of the stability operator $L$ on $\Sigma$ with Robin condition $B\phi=0$ whose associated eigenfunction can be chosen to be positive and the same holds for the formal adjoint $L^*$ on $\Sigma$ with Robin boundary condition $B^*\phi^*=0$. 

Thus, it follows from Lemma 3.6 of \cite{mendes2022rigidity} that there exist an outer neighborhood $V\cong [0,\varepsilon)\times \Sigma$ of $\Sigma\cong\{0\}\times \Sigma$ in $M^+$ and a positive function $\varphi=\varphi(t,x)$ defined on $V$ such that:
\begin{enumerate}
\item the metric $g$ restricted to $V$ has the orthogonal decomposition $g|_V=\varphi^2dt^2+\gamma_t$, where $\gamma_t$ is the induced metric on $\Sigma_t\cong \{t\}\times \Sigma$.
\item Each $\Sigma_t$ is a free boundary surface in $(M,g,K,E)$ with constant null mean curvature $\theta^+(t)$ with respect to the outward unir normal $N_t=\varphi_t^{-1}\frac{\partial}{\partial t}$, where $\varphi_t=\varphi(t,\cdot)$ and  $N_0=N$.
\item $\partial\varphi_t/\partial\nu_t=\Pi^{\partial M}(N_t,N_t)\varphi_t$ on $\partial\Sigma_t$, where $\nu_t$ is the outward unit normal of $\partial\Sigma_t$ in $(\Sigma_t,\gamma_t)$.
\end{enumerate}

In the following, to simplify the notation, we will denote $\theta^+(t)$ just by $\theta(t)$. By  \eqref{variationoftheta} and \eqref{stabilityoperator}, we have at all $t\in[0,\varepsilon)$ that

\begin{eqnarray*}
\dfrac{\theta^\prime}{\varphi_t}-\theta\tau &=&-\dfrac{\Delta_{\Sigma_t}\varphi_t}{\varphi_t} +2\left\langle X^t,\dfrac{\nabla_{\gamma_t}\varphi_t}{\varphi_t}\right\rangle+Q-|X^t|^2+\divergente_{\gamma_t}X-\frac{\theta^2}{2}\\
&\leq & \divergente_{\gamma_t} Y^t-|Y^t|^2+Q,
\end{eqnarray*}
where $\Delta_{\Sigma_t}$ and $\nabla_{\gamma_t}$ denote, respectively, the Laplacian operator and the gradient on $(\Sigma_t,\gamma_t)$, $X^t$ is the tangent vector field on $(\Sigma_t,\gamma_t)$ dual to the $1$-form $K(\cdot,N_t)$ and $Y^t=X^t-\nabla_{\gamma_t}(\ln \varphi_t)$.

Integrating over $\Sigma_t$ and using the divergence theorem and that $\theta^\prime$ and $\theta$ are constant on $\Sigma_t$, we have 

\begin{eqnarray*}
\theta^\prime\int_{\Sigma_t}\dfrac{1}{\varphi_t}\,da_t-\theta\int_{\Sigma_t}\tau\,da_t & \leq & \int_{\Sigma_t}Q\,da_t+\int_{\partial\Sigma_t}\langle Y^t,\nu_t\rangle\,d\ell_t\\
&=&\int_{\Sigma_t}K_{\Sigma_t}\,da_t - \int_{\Sigma_t}(\mu+J(N_t))\,da_t-\dfrac{1}{2}\int_{\Sigma_t}|\chi^t|^2\,da_t\\
&&+\int_{\partial\Sigma_t}\left(\langle X^t,\nu_t\rangle-\dfrac{1}{\varphi_t}\dfrac{\partial \varphi_t}{\partial\nu_t}\right)\,d\ell_t\\
&\leq &\int_{\Sigma_t}K_{\Sigma_t}\,da_t-\int_{\Sigma_t}(\Lambda+|E|^2)\,da_t\\
&&+\int_{\partial\Sigma_t}\left((i_{\bar{N}}\pi)^\top(N_t)-\Pi^{\partial M}(N_t,N_t)\right)\,d\ell_t\\
&\leq&\int_{\Sigma_t}K_{\Sigma_t}\,da_t-\Lambda|\Sigma_t|-\dfrac{4\pi ^2\mathfrak{q}(\Sigma_t)^2}{|\Sigma_t|}+\int_{\partial\Sigma_t}k_{g}\,d\ell_t\\
&&+\int_{\partial\Sigma_t}\left((i_{\bar{N}}\pi)^\top(N_t)-H^{\partial M}\right)\,d\ell.
\end{eqnarray*}
Thus, using the boundary energy condition $H^{\partial M}\geq |(i_{\bar{N}}\pi^\top)|\geq (i_{\bar{N}}\pi)^\top(N_t)$ and the Gauss-Bonnet theorem, we have that
\begin{equation}\label{ineqtheta1}
\theta^\prime\int_{\Sigma_t}\dfrac{1}{\varphi_t}\,da_t-\theta\int_{\Sigma_t}\tau\,da_t \leq  2\pi\chi(\Sigma_t)-\Lambda|\Sigma_t|-\dfrac{4\pi^2\mathfrak{q}(\Sigma_t)}{|\Sigma_t|}.
\end{equation}
We note that $\chi(\Sigma_t)=\chi(\Sigma)$ for all $t$. Moreover, since $E$ is tangent to $\partial M$,  $\divergente E=0$ and $\Sigma_t$ is homologous to $\Sigma_0=\Sigma$, for all $t$, we obtain by applying the divergence theorem that
$$
0=\int_{\Sigma_t}\langle E,N_t\rangle\,da_t-\int_{\Sigma}\langle E,N\rangle\,da=\mathfrak{q}(\Sigma_t)-\mathfrak{q}(\Sigma),
$$
that is,
$q(\Sigma_t)=\mathfrak{q}(\Sigma)$ for all $t$. 

Therefore, since $\Sigma$ attains equality in \eqref{ineqpolynomial2}, inequality \eqref{ineqtheta1} becomes
\begin{equation}
\theta^\prime\int_{\Sigma_t}\dfrac{1}{\varphi_t}\,da_t-\theta\int_{\Sigma_t}\tau\,da_t\leq \Lambda\left(|\Sigma|- |\Sigma_t|\right) +4\pi^2\mathfrak{q}(\Sigma)^2\left(\dfrac{1}{|\Sigma|}-\dfrac{1}{|\Sigma_t|}\right)
\end{equation}
for all $t\in[0,\varepsilon)$.

First, suppose that $\Lambda\geq 0$. Since $\Sigma$ is outer area minimizing if $\Lambda>0$, we have that $|\Sigma|\leq |\Sigma_t|$ for all $t\in [0,\varepsilon)$ in this case. Thus,

\begin{equation*}
\theta^\prime\int_{\Sigma_t}\dfrac{1}{\varphi_t}\,da_t-\theta\int_{\Sigma_t}\tau\,da_t  \leq  \dfrac{4\pi^2\mathfrak{q}(\Sigma)^2}{|\Sigma||\Sigma_t|}\left(|\Sigma_t|-|\Sigma|\right)
\leq C\left(|\Sigma_t|-|\Sigma|\right),  
\end{equation*}
where $C=\frac{4\pi^2\mathfrak{q}(\Sigma)^2}{|\Sigma|^2}$. Now, using the fundamental theorem of calculus, the first variation of area and the fact that $\Sigma_t$ is a free boundary surface, we obtain that
\begin{eqnarray*}
\theta^\prime\int_{\Sigma_t}\dfrac{1}{\varphi_t}\,da_t-\theta\int_{\Sigma_t}\tau\,da_t &\leq &C\int_0^t\dfrac{d}{ds}|\Sigma_s|\,ds\\
& = & C\int_0^t\left(\int_{\Sigma_s}H(s)\varphi_s\,da_s\right)
\,ds,
\end{eqnarray*}
where $H(s)$ denotes the mean curvature of $\Sigma_t$ with respect to $N_s$. 

Thus, since $\theta(s)=\tr_{\Sigma_s}K+H(s)\geq H(s)$ because $K$ is $2$-convex, we have that
\begin{equation}\label{ineq1thetamain}
\theta^\prime(t)\int_{\Sigma_t}\dfrac{1}{\varphi_t}\,da_t-\theta(t)\int_{\Sigma_t}\tau\,da_t \leq C\int_0^t\theta(s)\left(\int_{\Sigma_s}\varphi_s\,da_s\right)\,ds,
\end{equation}
for all $t\in [0,\varepsilon)$, where we have used that $\theta(s)$ is constant on $\Sigma_s$. 

Next, note that $\theta(t)\geq 0$ for all $t\in[0,\varepsilon)$ since $\Sigma$ is weakly outermost. Thus, we can apply Lemma \ref{Lemma f nonnegative} to \eqref{ineq1thetamain} in order to conclude that $\theta(t)=0$ for all $t\in[0,\varepsilon)$, that is, $\Sigma_t$ is a MOTS with free boundary for all $t\in[0,\varepsilon)$. 

As a consequence, we have that all the above inequalities are in fact equalities. In particular, $\tr_{\Sigma_t} K=0$ in $\Sigma_t$ and $H(t)\leq 0$, for all $t\in[0,\varepsilon)$. Therefore, $|\Sigma_t|\leq |\Sigma_0|=|\Sigma|$ for all $t\in[0,\varepsilon)$. Since $\Sigma$
is outer area minimizing, we have that $|\Sigma_t|=|\Sigma|$ for all $t$. This implies that $\Sigma_t$ is a minimal MOTS. Moreover, we also obtain that $\Sigma_t$ attains equality in \eqref{ineqlambdapos} or \eqref{ineqlambdazero}, as the case may be. Therefore, for all $t\in[0,\varepsilon)$, the conclusion of the infinitesimal rigidity in Proposition \ref{infrigidity} can be applied to $\Sigma_t$. 

In particular, as $\mathfrak{q}(\Sigma_t)=\mathfrak{q}(\Sigma)$ and $|\Sigma_t|=|\Sigma|$ for all t, we have that $E=aN_t$ on $V\cong[0,\varepsilon)\times \Sigma$, for some constant $a\in\mathbb{R}$. 

Next, since $K_{\Sigma_t}=\mu-|J|$ on $\Sigma_t$ and $\Sigma_t$ is minimal for all $t$, we have, by using the Gauss equation, that
\begin{eqnarray*}
0=-\frac{H^\prime(t)}{\varphi_t}&=&\frac{\Delta_{\Sigma_t}\varphi_t}{\varphi_t}+\frac{1}{2}\left(R_M+|A_{\Sigma_t}|^2\right)-K_{\Sigma_t} \\
&=&\frac{\Delta_{\Sigma_t}\varphi_t}{\varphi_t}+\frac{1}{2}\left(R_M+|A_{\Sigma_t}|^2\right)-(\mu-|J|)\\
&=&\frac{\Delta_{\Sigma_t}\varphi_t}{\varphi_t}+\frac{|A_{\Sigma_t}|^2}{2}+\frac{|K|^2}{2}-\frac{(\tr K)^2}{2}+|J|\\
&=&\frac{\Delta_{\Sigma_t}\varphi_t}{\varphi_t}+\frac{|A_{\Sigma_t}|^2}{2}+\frac{|K|_{T\Sigma_t}|^2}{2}+|X^t|^2+|J|,
\end{eqnarray*}
where in the last equality we have used that $\tr_{\Sigma_t}K=0$.

Integrating over $\Sigma_t$ and using that
$$
\dfrac{\partial\varphi_t}{\partial\nu_t}=\Pi^{\partial M}(N_t,N_t)\varphi_t=|(i_{\bar{N}}\pi)^\top|\varphi_t\geq 0
$$
we obtain
\begin{eqnarray*}
0&\leq& \int_{\partial\Sigma_t}|(i_{\bar{N}}\pi)^\top|+ \int_{\Sigma_t}\left(\frac{|\nabla_{\gamma_t}\varphi_t|^2}{\varphi_t}+\frac{|A_{\Sigma_t}|^2}{2}+\frac{|K|_{T\Sigma_t}|^2}{2}+|X^t|^2+|J|\right)=0.
\end{eqnarray*}
Therefore, we conclude that $\Pi^{\partial M}(N_t,N_t)=|(i_{\bar{N}} \pi)^\top|=0$ on $\partial\Sigma_t$, $\varphi_t$ depends only on $t$ and $\Sigma_t$ is totally geodesic, $J|_{\Sigma_t}=0$, $K|_{T\Sigma_t}=0$ and $X^t=0$ for all $t\in[0,\varepsilon)$. In particular, $K=fdt^2$ for some smooth function $f$ defined on $V\cong [0,\varepsilon)\times \Sigma$.  Since $J|_{\Sigma_t}=0$ for all $t$. Note that $J|_{\Sigma_t}=0$ for all $t$ is equivalent to say that $d(\tr K) = \divergente K$ on $V$. Since $K|_{T\Sigma_t}=0$ and $K(\cdot,N_t)=0$, for all $t$, it follows that $\tr K$ depends only on $t$ and, as $\varphi_t$ is constant on $\Sigma_t$, the same holds for $f$.

Now, as $\Sigma_t$ is totally geodesic for all $t$, we conclude that $\gamma_t=\gamma_0$ for all $t$. Moreover, since $\varphi_t=\varphi(t)>0$ depends only on $t$, we can assume, after a change of variable, that $\varphi_t=1$ for all $t\in[0,\varepsilon)$. In particular, the metric $g$ can be written on $V\cong [0,\varepsilon)\times \Sigma$ as  $g=dt^2+\gamma_0$.

Finally, suppose that $\Lambda<0$. In this case, we have that

\begin{eqnarray*}
\theta^\prime\int_{\Sigma_t}\dfrac{1}{\varphi_t}\,da_t-\theta\int_{\Sigma_t}\tau\,da_t
&\leq & \Lambda\left(|\Sigma|-|\Sigma_t|\right)+\dfrac{4\pi^2 \mathfrak{q}(\Sigma)^2}{|\Sigma||\Sigma_t|}\left(|\Sigma_t|-|\Sigma|\right)\\
&\leq &C(t)\left(|\Sigma_t|-|\Sigma|\right),
\end{eqnarray*}
where $C(t)=-\Lambda+\dfrac{4\pi^2\mathfrak{q}(\Sigma)^2}{|\Sigma||\Sigma_t|}>0$, for all $t\in[0,\varepsilon)$.

Therefore, in a manner analogous to the case $\Lambda\geq 0$, we can obtain
$$
\theta^\prime(t) \eta(t) - \theta(t)\rho(t) \leq \int_0^t \theta(s)\xi(s)\,ds, \ \ t\in[0,\varepsilon), 
$$
where
$$
\eta(t)=\frac{1}{C(t)}\int_{\Sigma_t} \frac{1}{\varphi_t}\,da_t, \ \ \rho(t)=\frac{1}{C(t)}\int_{\Sigma_t}\tau\,da_t \ \ \mbox{and} \ \ \xi(t)=\int_{\Sigma_t}\varphi_t\,da_t.
$$

Thus, since $\theta(t)\geq 0$ for all $t\in[0,\varepsilon)$, another application of Lemma \ref{Lemma f nonnegative} yields $\theta(t)=0$ for all $t\in [0,\varepsilon)$. From this point on, we proceed in the same way as in the case $\Lambda\geq 0$.

\subsection*{Acknowledgments} I. N. was partially supported by CNPq, Grants No. 444531/2024-6, 403869/2024-2, and 400078/2025-2. I. N. and E. F. were partially supported by CAPES through the PRAPG Program, Grant No. 88887.986826/2024-2. 
Part of this work was carried out during the postdoctoral fellowship of E. F. funded by CAPES through the PRAPG Program at the Graduate Program in Mathematics (PPGMAT) of the Federal University of Maranhão (UFMA). E. F. would like to thank PPGMAT - UFMA for its hospitality.


\end{document}